\documentclass{amsart}
\usepackage{amssymb}
\usepackage{latexsym}

\def\be{\begin{equation}}
\def\ee{\end{equation}}

\def\bm{\begin{matrix}}
\def\em{\end{matrix}}

\def\hs{{\mathrm{HS}}}

\def\SL{{\mathrm {SL}}}

\newcommand{\Z}{{\mathbb Z}}
\newcommand{\R}{{\mathbb R}}
\newcommand{\Q}{{\mathbb Q}}
\newcommand{\C}{{\mathbb C}}
\newcommand{\D}{{\mathbb D}}
\renewcommand{\H}{{\mathbb H}}

\newcommand{\dist}{{\mathrm{dist}}}
\newcommand{\tr}{{\mathrm{Tr}}}

\newtheorem*{main}{Main Theorem}
\newtheorem{thm}{Theorem}
\newtheorem{cor}{Corollary}

\newtheorem{lemma}{Lemma}[section]

\begin{document}

\title[Absolute Continuity of the IDS for the AMO]
{Absolute Continuity of the Integrated Density of States for the Almost Mathieu Operator
with Non-Critical Coupling}

\author{Artur Avila and David Damanik}

\address{
CNRS UMR 7599, Laboratoire de Probabilit\'es et Mod\`eles al\'eatoires\\
Universit\'e Pierre et Marie Curie--Boite courrier 188\\
75252--Paris Cedex 05, France } \email{artur@ccr.jussieu.fr}

\address{
Department of Mathematics\\
Rice University\\
Houston, TX 77005, USA } \email{damanik@rice.edu}

\thanks{This research was partially conducted during the period A.\ A.\ served as a Clay Research Fellow.
D.\ D.\ was supported in part by NSF grants DMS--0500910 and
DMS--0653720. He is grateful to Rowan Killip for illuminating
conversations. Parts of this work were done during the
Mini-Workshop ``Dynamics of Cocycles and One-Dimensional Spectral
Theory,'' which was held in Oberwolfach in November 2005. The
authors thank Svetlana Jitomirskaya for several discussions and
the anonymous referee for comments that led to improvements of the
paper.}

\begin{abstract}
We show that the integrated density of states of the almost
Mathieu operator is absolutely continuous if and only if the
coupling is non-critical.  We deduce for subcritical coupling that
the spectrum is purely absolutely continuous for almost every
phase, settling the measure-theoretical case of Problem 6 of Barry
Simon's list of Schr\"odinger operator problems for the
twenty-first century.
\end{abstract}

\maketitle

\section{Introduction}

This work is concerned with the almost Mathieu operator
$H=H_{\lambda,\alpha,\theta}$ defined on $\ell^2(\Z)$ \be (H
u)_n=u_{n+1}+u_{n-1}+2 \lambda \cos(2\pi[\theta+n \alpha]) u_n \ee
where $\lambda \neq 0$ is the coupling, $\alpha \in \R \setminus
\Q$ is the frequency, and $\theta \in \R$ is the phase.  This is
the most heavily studied quasiperiodic Schr\"odinger operator,
arising naturally as a physical model (see \cite {L2} for a recent
historical account and for the physics background).

We are interested in the integrated density of states, which can
be defined as the limiting distribution of eigenvalues of the
restriction of $H=H_{\lambda,\alpha,\theta}$ to large finite
intervals.  This limiting distribution (which exists by \cite
{AS2}) turns out to be a $\theta$ independent continuous
increasing surjective function $N=N_{\lambda,\alpha}:\R \to
[0,1]$.  The support of the probability measure $dN$ is precisely
the spectrum $\Sigma=\Sigma_{\lambda,\alpha}$. Another important
quantity, the Lyapunov exponent $L=L_{\lambda,\alpha}$, is
connected to the integrated density of states by the Thouless
formula $L(E)=\int \ln |E-E'| \, dN(E')$.

Since $\Sigma$ is a Cantor set (\cite {BS}, \cite {CEY},
\cite{L1}, \cite {P}, \cite {AK1}, \cite {AJ1}), $N$ is ``Devil's
staircase''-like. It is also known that $\Sigma$ has Lebesgue
measure $|4-4|\lambda||$ \cite {AK1}. See \cite{L2} for the
history of these two problems and additional references. The
spectrum has zero Lebesgue measure precisely when $|\lambda|=1$,
the critical coupling.

Recently there was quite a bit of interest in the regularity
properties of the integrated density of states.  The modulus of
continuity is easily seen to be quite poor for generic
frequencies, so positive results (such as H\"older continuity
\cite {GS1}) in this direction have depended on suitable (full
measure) conditions on the frequency. There was more hope in
proving positive general results on absolute continuity of $N$
since \cite {L} established that for almost every $\alpha$ and
$\lambda$ non-critical (or every $\alpha$ and certain values of
$\lambda$), $dN$ has an absolutely continuous component (for
critical coupling, $dN$ never has an absolutely continuous
component due to zero Lebesgue measure of $\Sigma$).  Later, this
result was improved to full absolute continuity under a full
measure condition on $\alpha$ and all non-critical couplings \cite
{J} (the condition on $\alpha$ was later improved in \cite {AJ1}).
Let us also mention the recent work of Goldstein-Schlag \cite
{GS2} that establishes absolute continuity of the integrated
density of states in the regime of positive Lyapunov exponent
under a full measure condition on $\alpha$, but for more general
potentials.

In this work, we prove the following result, which completely
describes the set of parameter values for which the integrated
density of states is absolutely continuous.

\begin{main}

The integrated density of states of $H_{\lambda,\alpha,\theta}$ is
absolutely continuous if and only if $|\lambda| \neq 1$.

\end{main}

We point out that prior to this result, it was unknown whether there could be parameters
for which the spectrum contains pieces (non-empty open subsets) of both positive and zero
Lebesgue measure.

Our Main Theorem has an important consequence regarding the nature
of the spectral measures of the operators
$H_{\lambda,\alpha,\theta}$.  For $|\lambda| \geq 1$, it is known
that the spectral measures have no absolutely continuous
component. For $|\lambda|<1$, one expects the spectral measures to
be absolutely continuous. Indeed, the belief in such a simple and
general description of the nature of the spectral measures dates
back to the fundamental work of Aubry-Andr\'e \cite {AA}.\footnote
{In \cite {AA}, it is actually proposed that spectral measures are
absolutely continuous for $|\lambda|<1$ and atomic for
$|\lambda|>1$ (with both regimes being linked in the heuristic
reasoning).  The problem turned out to be very subtle, and the
claim for $|\lambda|>1$ was soon shown to be wrong as stated (see
further discussion below).  This discovery did generate some
doubts regarding the claim for $|\lambda|<1$ (see Problem 5 in
Section 11 of \cite {s3}) before optimism was regained with the
work of Last \cite {L}.} More recently, this conjecture shows up
as Problem~6 of Simon's list of open problems in the theory of
Schr\"odinger operators for the twenty-first century \cite {s2}.
As we will discuss in more detail below, the result is known for
$\alpha$'s satisfying a Diophantine condition. The strategy of the
proof, however, clearly does not extend to the case of $\alpha$'s
that are well approximated by rational numbers. Indeed, Simon
points out in \cite{s2} that ``one will need a new understanding
of absolutely continuous spectrum to handle the case of Liouville
$\alpha$'s.''

A beautiful result of Kotani \cite {k2}, which has not yet
received the attention and exposure it deserves, shows that if the
Lyapunov exponent vanishes in the spectrum, then absolute
continuity of the IDS is equivalent to absolute continuity of the
spectral measures for almost every $\theta$; see also the survey
\cite{D} of Kotani theory and its applications. By \cite {BJ}, if
the coupling is subcritical, the Lyapunov exponent is zero on the
spectrum.  We therefore obtain the following corollary.

\begin{cor}

If $|\lambda|<1$, then the spectral measures of
$H_{\lambda,\alpha,\theta}$ are absolutely continuous for almost
every $\theta$.

\end{cor}

This settles Problem~6 of \cite{s2}, at least in an almost
everywhere sense. 
\footnote {After this work was completed, an approach to proving
absolutely continuous spectrum for every phase has been proposed
by the first named author. The proposed solution does use (as an
important step) the techniques developed here.}

\medskip

As mentioned before, the Main Theorem had been established already
for certain ranges of parameters.  Let us discuss in more detail
which range of parameters could be covered by such methods, and
which range of parameters will be treated by the new techniques
introduced in this paper.

Due to the symmetries of the system, we may restrict our attention to
$\lambda>0$.  It is known that
$N_{\lambda,\alpha}(E)=N_{\lambda^{-1},\alpha}(\lambda^{-1}
E)$ (Aubry duality).  Thus in order to establish the Main Theorem it is
enough to show that $N_{\lambda,\alpha}$ is absolutely continuous for
$0<\lambda<1$.

In the ``Diophantine case,'' the following argument has been successful. One establishes
pure point spectrum for almost every phase for $H_{\lambda^{-1},\alpha,\theta}$, which
implies by the strong version of
Aubry duality \cite {GJLS} that the spectrum of $H_{\lambda,\alpha,\theta}$ is
absolutely continuous for almost every phase, which obviously implies that
$N_{\lambda,\alpha}$ is absolutely continuous.

Pure point spectrum for $H_{\lambda^{-1},\alpha,\theta}$ for
almost every phase was indeed established by Jitomirskaya for all
$\lambda^{-1}>1$ under a full measure condition on $\alpha$; see
\cite {J}. This result was recently strengthened by Avila and
Jitomirskaya as follows. Let $p_n/q_n$ be the continued fraction
approximants to $\alpha$ and let \be \label {alphaqn}
\beta=\beta(\alpha)=\limsup_{n \to \infty} \frac {\ln q_{n+1}}
{q_n}. \ee If $\beta = 0$ and $\lambda^{-1}>1$, then
$H_{\lambda^{-1},\alpha,\theta}$ has pure point spectrum for
almost every phase; see \cite[Theorem~5.2]{AJ1}. As a consequence,
the main theorem is known in the case $\beta = 0$.

However, it has been understood for a long time that this approach
cannot work for all $\alpha$ \cite {AS1}. Indeed, if $\beta > 0$
and $1<\lambda^{-1}<e^\beta$, Gordon's Lemma \cite {G} (and the
formula for the Lyapunov exponent \cite {BJ}) shows that there is
no point spectrum at all!  It is expected that pure point spectrum
for almost every phase does hold when $\lambda^{-1}>e^\beta$. This
is currently established when $\lambda^{-1}>e^{16 \beta/9}$ \cite
{AJ1}. In any event, the main theorem cannot be proven
via the ``localization plus duality'' route when $\beta > 0$.

In this work, we provide a new approach to absolute continuity of
the IDS via rational approximations.  This approach works when
$0<\lambda<1$ and $\beta>0$.  A feature of this approach is that
absolutely continuous spectrum for almost every phase is obtained
as a consequence of the absolute continuity of the IDS, rather
than the other way around as is usually the case in the regime of
zero Lyapunov exponent. In other words, our work establishes the
first application of Kotani's gem (a sufficient condition for
purely absolutely continuous spectrum in terms of the IDS and the
Lyapunov exponent) from \cite{k2}.

\section{Preliminaries}

\subsection{$\SL(2,\R)$-action}

Recall the usual action of $\SL(2,\C)$ on $\overline \C$,
$$
\left ( \bm a & b \\ c & d \em \right ) \cdot z=\frac {az+b}
{cz+d}.
$$

In the following we restrict to matrices $A \in \SL(2,\R)$.

Such matrices preserve $\H=\{z \in \C,\, \Im z>0\}$. The
Hilbert-Schmidt norm of
$$
A=\left (\bm a & b \\ c & d \em \right )
$$
is $\|A\|_\hs=(a^2+b^2+c^2+d^2)^{1/2}$.  Let $\phi(z)=\frac
{1+|z|^2} {2 \Im z}$ for $z \in \H$.  Then
\begin{align*}
\phi(A \cdot i ) & = \phi ( \tfrac{ai + b}{ci + d} ) \\
& = \frac {1+|\frac{ai + b}{ci + d}|^2} {2 \Im \frac{ai + b}{ci + d}}\\
& = \frac {1+\frac{a^2 + b^2}{c^2 + d^2}} {2 \Im \frac{(ai + b)(-ci + d)}{c^2 + d^2}}\\
& = \frac {a^2 + b^2 + c^2 + d^2} {2 (ad-bc)}
\end{align*}
and hence $\|A\|_\hs^2=2 \phi(A \cdot i)$. Thus $\phi(z)$ is half the square of the
Hilbert-Schmidt norm of an $\SL(2,\R)$ matrix that takes $i$ to $z$.

The rotation matrices
\be
R_\theta=\left ( \bm \cos 2 \pi \theta & -\sin 2 \pi \theta \\
\sin 2 \pi \theta & \cos 2 \pi \theta \em \right ) \ee are characterized by the fact that
they fix $i$. One easily checks that $\|R_\theta A\|_\hs=\|AR_\theta\|_\hs=\|A\|_\hs$. In
particular $\phi(R_\theta z)=\phi(z)$.

We notice that $\phi(z) \geq 1$, $\phi(i)=1$ and $|\ln \phi(z)-\ln \phi(w)| \leq
\dist_\H(z,w)$ where $\dist_\H$ is the hyperbolic metric on $\H$, normalized so that
$\dist_\H(a i , i)=|\ln a|$ for $a > 0$.

If $|\tr A|<2$, then there exists a unique fixed point $z \in \H$, $A \cdot z=z$.  Let
$0<\rho<1/2$ be such that $\tr A=2 \cos 2 \pi \rho$.  Let us show that
\begin{equation}\label{2 sin}
\phi(z)=\frac {1} {2 \sin 2 \pi \rho} (\|A\|^2_\hs-2 \cos 4 \pi \rho)^{1/2},
\end{equation}
so that
\begin{equation}\label{phi}
\phi(z) \leq \frac {\sqrt{2} \|A\|_\hs} {2 \sin 2 \pi \rho}.
\end{equation}
Let $B \in \SL(2,\R)$ be such that $B \cdot i=z$. Then \be A=B R_{\pm \rho} B^{-1}
\ee since $B^{-1}A B$ fixes $i$ and hence is a rotation that has the same trace as $A$.
Write $B=R D$ with $R$ a rotation and $D$ diagonal, $$D=\left ( \bm \lambda & 0
\\ 0 & \lambda^{-1} \em \right ).$$ (First stretch $i$ suitably and then rotate it to $z$.)
Then $D R_{\pm \rho} D^{-1}$ has a unique fixed point $R^{-1} \cdot z=D \cdot i$ and
$$\phi(z)=\phi(R^{-1} \cdot z)=(\lambda^2+\lambda^{-2})/2=\|D\|^2_\hs/2.$$ On the other
hand,
\begin{align*}
\|A\|^2_\hs & =\|D R_{\pm \rho} D^{-1}\|^2_\hs \\
& = 2 \cos^2 2 \pi \rho+(\lambda^4+\lambda^{-4}) \sin^2 2 \pi \rho \\
& = 2 \cos^2 2 \pi \rho +(\|D\|^4_\hs-2) \sin^2 2 \pi \rho \\
& = 2 \cos 4 \pi \rho + \|D\|^4_\hs \sin^2 2 \pi \rho\\
& = 2 \cos 4 \pi \rho + 4 \phi(z)^2 \sin^2 2 \pi \rho,
\end{align*}
and we obtain \eqref {2 sin}.

\subsection{Lyapunov Exponent}

Fix $\lambda \in \R$, $\alpha \in \R$, $E \in \R$.  Let \be
A(\theta)=A^{(\lambda,E)}(\theta)= \left (\bm E-2\lambda \cos 2 \pi \theta & -1 \\ 1 & 0
\em \right ). \ee Denote \be A_n(\theta)=A^{(\lambda,\alpha,E)}_n(\theta)= A(\theta+(n-1)
\alpha) \cdots A(\theta). \ee The Lyapunov exponent $L=L_{\lambda,\alpha}$ is defined as
\be L(E)= \lim_{n \to \infty} \frac {1} {n} \int_{\R/\Z} \ln \|A_n(\theta)\| \, d\theta.
\ee By unique ergodicity of irrational rotations, for $\alpha \in \R \setminus \Q$, we
have $L(E)=\lim \frac {1} {n} \sup_\theta \ln \|A_n(\theta)\|$.

\begin{thm}[\cite {BJ}, Corollary 2] \label {continuity}

If $0<\lambda<1$, $\alpha \in \R \setminus \Q$ and
$E \in \Sigma_{\lambda,\alpha}$ then $L(E)=0$.

\end{thm}

\subsection{The Almost Mathieu Operator}

We extend the definitions of the introduction to the case of rational frequencies
$p/q$ (here and in what follows, $p/q$ will always denote the reduced fraction
of a rational number).  We
can of course define $H_{\lambda,p/q,\theta}$ by the same formula as in the introduction;
however, the several related quantities are no longer $\theta$ independent.

Let $N_{\lambda,p/q}$ be the average over $\theta$ of the integrated density of states of
the operators $H_{\lambda,p/q,\theta}$. Then we still have the Thouless formula \be
L(E)=\int \ln |E'-E| dN(E'). \ee

Let $\Sigma_{\lambda,p/q}$ be the union over $\theta$ of the spectra of
$H_{\lambda,p/q,\theta}$ and let $\sigma_{\lambda,p/q}$ be the intersection over $\theta$
of the spectra of $H_{\lambda,p/q,\theta}$.

With those definitions, $\Sigma_{\lambda,\alpha}$ and $N_{\lambda,\alpha}$ are continuous
in $\lambda$ and $\alpha$.

\begin{thm}[\cite{AMS}, Proposition~7.1] \label {ams}

The Hausdorff distance between $\Sigma_{\lambda,\alpha}$ and
$\Sigma_{\lambda,\alpha'}$ is at most $6 (2 \lambda)^{1/2}
|\alpha-\alpha'|^{1/2}$, for $|\alpha-\alpha'| \leq \frac {C} {\lambda}$,
where $C>0$ is some constant.

\end{thm}

\subsection{Periodic Case} \label {periodic}

When $\alpha=p/q$, $\tr A_q(\theta)$ is periodic of period $1/q$. Since $\tr A_q(\theta)=
-\lambda^q e^{2 \pi i q \theta} - \lambda^q e^{-2 \pi i q \theta}+\sum_{j=1-q}^{q-1} a_j
e^{2 \pi i j \theta}$ it follows that $a_j=0$ for $0<|j|<q$ and we have the Chambers
formula \be \tr A_q(\theta)= -2 \lambda^q \cos 2 \pi q \theta+a_0 \ee where
$a_0=a_0(\lambda,p/q,E)$.

We have $\Sigma_{\lambda,p/q}=\{E : \inf_\theta |\tr A_q(\theta)|
\leq 2\}$ and $\sigma_{\lambda,p/q}=\{E : \sup_\theta |\tr
A_q(\theta)| \leq 2\}$.  Bands of $\Sigma_{\lambda,p/q}$ are
closures of the connected components of $\{E : \inf_\theta |\tr
A_q(\theta)|<2\}$. Then there are $q$ bands,
$\Sigma_{\lambda,p/q}$ is the union of the bands, the bands can
only touch, possibly, at the edges. Moreover,
$\sigma_{\lambda,p/q}$ is non-empty if and only if $0<\lambda \leq
1$, in which case it has $q$ connected components and each band of
$\Sigma_{\lambda,p/q}$ intersects in its interior a single
connected component of $\sigma_{\lambda,p/q}$.  (The facts above
can be all deduced from Chambers formula, see \cite {AMS} or
\cite{BS}.)

\begin{thm}[\cite {AMS}, Theorems 1 and 2] \label {measure}

For $0<\lambda<1$, we have $|\sigma_{\lambda,p/q}|=4-4\lambda$ and $|\Sigma_{\lambda,p/q}
\setminus \sigma_{\lambda,p/q}| \leq 4 \pi \lambda^{q/2}$.

\end{thm}

Passage to the limit at irrational frequencies yields the lower
bound of Thouless \cite {T} \be \label {4-4lambda}
|\Sigma_{\lambda,\alpha}| \geq 4-4\lambda, \quad 0<\lambda<1,\,
\alpha \in \R \setminus \Q \ee (of course, equality is now known
to hold in the above formula, but we will not need it).

\subsubsection{Formulas for the IDS Inside a Band}

If $|\tr A_q(\theta)| \leq 2$, let $0 \leq \rho(\theta) \leq 1/2$
be such that $\tr A_q(\theta)=2 \cos (2\pi \rho(\theta))$. Let
also $\rho(\theta)=0$ if $\tr A_q(\theta)>2$ and
$\rho(\theta)=1/2$ if $\tr A_q(\theta)<-2$. Let $\rho$ be the
average over $\theta$ of $\rho(\theta)$.  Then if $E$ belongs to
the $k$-th band of $\Sigma_{\lambda,p/q}$, we have the formula \be
\label {n rho} q N(E)=k-1+(-1)^{q+k-1} 2 \rho+\frac
{1-(-1)^{q+k-1}} {2}. \ee This formula is immediate from the
relation between the integrated density of states and the fibered
rotation number; see \cite{AS2} and \cite {JM}.

If $|\tr A_q(\theta)|<2$, let $m(\theta)$ be the fixed point of
$A_q(\theta)$ in $\H$. Note that, by periodicity, we have
$A(\theta) m(\theta) = m(\theta + p/q)$. Moreover, \be \label
{dnp} \frac {d} {d E} N(E)= \frac{1}{2\pi} \int_{|\tr
A_q(\theta)|<2} \phi(m(\theta)) \, d\theta. \ee This formula can
be obtained for instance as a very simple case of the general
formulas for absolutely continuous spectrum of \cite {DS}.

\subsection{Derivative of the IDS in the Irrational Case}

The following result is a consequence of Theorem \ref {continuity}
and \cite {DS,s1}.

\begin{thm} \label {Y}

Let $0<\lambda<1$, $\alpha \in \R \setminus \Q$.  Then there
exists a full Lebesgue measure subset $Y \subset
\Sigma_{\lambda,\alpha}$ such that for every $E \in Y\!$, there
exists a measurable function $\tilde m:\R/\Z \to \H$ such that
$A(\theta) \cdot \tilde m(\theta)=\tilde m(\theta+\alpha)$ and \be
\label{dnp2} \frac {d} {dE} N_{\lambda,\alpha}(E)= \frac{1}{2\pi}
\int \phi(\tilde m(\theta)) \, d\theta. \ee

\end{thm}

We refer the reader to \cite{D,k2} for more information on the
theory leading to the formulae \eqref{dnp} and \eqref{dnp2}; see
especially \cite[Theorem~5]{D} and \cite[Theorem~4.8]{k2}.

\section{Proof of the Main Theorem}

As discussed in the introduction, it is enough to prove the
following result. Recall the definition \eqref {alphaqn} of
$\beta(\alpha)$.

\begin{thm} \label {real}

Let $0<\lambda<1$, $\alpha \in \R \setminus \Q$.  If $\beta(\alpha)>0$, then the IDS is
absolutely continuous.

\end{thm}

The proof of this theorem will take up the remainder of this section. Throughout the
proof, $\lambda$ and $\alpha$ will be fixed.

Let $Y \subset \Sigma_{\lambda,\alpha}$ be as in Theorem \ref {Y}.  If $E \in Y$, let
$\tilde m$ be as in Theorem \ref {Y}.  It is enough to prove that \be \label {int_Y}
\int_Y \frac {d} {dE} N_{\lambda,\alpha} \, dE=1. \ee

The hypothesis implies that
\be
\left |\alpha - \frac{p}{q} \right | < e^{-(\beta-o(1)) q}
\ee
for arbitrarily large $q$. Fix some $p/q$ with this property and $q$
large.

For a fixed energy $E$, write $A=A^{(\lambda,E)}$, $A_n=A^{(\lambda,p/q,E)}_n$ and
$\tilde A_n=A^{(\lambda,\alpha,E)}_n$.

Let
$$
c = \min \{\beta/2,-\ln \lambda/2\}.
$$
Notice that $\max \{ e^{-\beta q} , \lambda^q \} \le e^{-2cq}$. Let
\begin{align*}
P_q = \{ & \rho \in [1/q,1/2-1/q] : \exists a=a(\rho), b=b(\rho) \text{ positive integers
with $a$ odd, } \\
& e^{c q/4}<b<e^{c q/2}, \text{ and } |4b \rho-a|<10/b \}.
\end{align*}
Define
$$
X = X_{\lambda,p/q} = \{ E \in \sigma_{\lambda,p/q} : \rho \in P_q \}.
$$
Notice that
\begin{align*}
|4 b(\rho) \rho(\theta)-a(\rho)| & = |4 b(\rho) \rho-a(\rho)| + |4 b(\rho)| |\rho(\theta)
- \rho| \\
& \lesssim e^{-c q/4} + e^{c q/2} \lambda^q \\
& \lesssim e^{-c q/4} + e^{c q/2} e^{-2cq} \\
& = O(e^{-cq/4})
\end{align*}
(where $|\rho-\rho(\theta)|$ is estimated using Chambers formula).

\begin{lemma} \label {NX}
We have \be |N_{\lambda,p/q} (X_{\lambda,p/q})| = 1 - o(1). \ee
\end{lemma}

\begin{proof}
By the Chambers formula (see \S \ref {periodic}) and \eqref {n rho},
$$
|N_{\lambda,p/q}(\Sigma_{\lambda,p/q} \setminus \sigma_{\lambda,p/q})| = o(1),
$$
so it suffices to show
$$
|N_{\lambda,p/q} (\sigma_{\lambda,p/q} \setminus X_{\lambda,p/q})| = o(1).
$$
This in turn follows once we show in each connected component $B$ of
$\sigma_{\lambda,p/q}$ that
$$
|N_{\lambda,p/q}(B \setminus X_{\lambda,p/q})| = o(1/q).
$$
Now,
$$
|N_{\lambda,p/q}(B \setminus X_{\lambda,p/q})| \le \frac{1 - 2|P_q|}{q}
$$
by \eqref {n rho}. Thus, it suffices to show that $|P_q| \to 1/2$.

This follows from the observation that numbers $\rho \in [0,1/2]$ such that the
denominators $d_k$ of the best approximants of $4 \rho$ satisfy $d_{k+1} < d_k^{4/3}$ for
all $k$ large enough have full Lebesgue measure on $[0,1/2]$ and belong to $\liminf_{q
\to \infty} P_q$. Thus, at least two denominators fall into the allowed window and one of
them can be used due to the fact that two successive numerators cannot both be even. If
there were two consecutive even numerators, then by the recursion all earlier numerators
must have been even; but the first one was $1$ and hence odd.
\end{proof}

\begin{lemma}

We have
\be
| \Sigma_{\lambda,p/q} \setminus \Sigma_{\lambda,\alpha} | \leq
e^{-(c-o(1))q}.
\ee
In particular,
\be \label {inp}
| X_{\lambda,p/q} \setminus \Sigma_{\lambda,\alpha} | \leq e^{-(c-o(1))q}.
\ee
\end{lemma}

\begin{proof}
The Lebesgue measure of $\Sigma_{\lambda,p/q}$ is $4-4\lambda+O(e^{-c q})$ by Theorem
\ref {measure}.  Since $\Sigma_{\lambda, p/q}$ is the union of $q$ intervals, Theorem
\ref {ams} implies \be | \Sigma_{\lambda, \alpha} \setminus \Sigma_{\lambda, p/q} |=O(q
e^{-cq}). \ee Since $| \Sigma_{\lambda, \alpha} | \ge 4 - 4\lambda$ by \eqref
{4-4lambda},
\begin{align*}
| \Sigma_{\lambda,p/q} \setminus \Sigma_{\lambda,\alpha} | & = | \Sigma_{\lambda,p/q}
 | - | \Sigma_{\lambda,p/q} \cap \Sigma_{\lambda,\alpha} | \\
& = | \Sigma_{\lambda,p/q} | - ( | \Sigma_{\lambda, \alpha} | - | \Sigma_{\lambda,
\alpha} \setminus \Sigma_{\lambda, p/q} | ) \\
& \le 4-4\lambda+O(e^{-c q}) - (4 - 4\lambda) + O(q e^{-cq}),
\end{align*}
and the result follows.\end{proof}

If $E$ belongs to the interior of $\sigma_{\lambda,p/q}$, let $m(\theta)$ be the fixed
point of $A_q(\theta)$ in $\H$, as in \S \ref {periodic}.

\begin{lemma} \label {o(q)}
We have \be \sup_{E \in X} \sup_\theta \ln \phi(m(\theta))=o(q). \ee
\end{lemma}

\begin{proof}
By Theorem~\ref{continuity}, the Lyapunov exponent is zero in
$\Sigma_{\lambda,\alpha}$.  By unique ergodicity of rotations (see
the comment before Theorem~\ref{continuity}), this means that for
every $E \in \Sigma_{\lambda,\alpha}$ and for every
$\varepsilon>0$, there exists $n_0(\varepsilon,E)$ such that
$$
\ln \|A^{(\lambda,\alpha,E)}_n(\theta)\|<\varepsilon n
$$
for every $\theta$ and every $E \in \Sigma_{\lambda,\alpha}$ for
$n>n_0(\varepsilon,E)$.  This obviously implies that there exists
$\delta(\varepsilon,E)>0$ such that if
$|\alpha'-\alpha|<\delta(\varepsilon,E)$ and
$|E'-E|<\delta(\varepsilon,E)$, then
$$
\ln \|A^{(\lambda,\alpha',E')}_n(\theta)\|<\varepsilon n
$$
for every $\theta$ and every $n_0(\varepsilon,E)< n \leq 2
n_0(\varepsilon,E)+1$, and hence, by subadditivity, for every
$n>n_0(\varepsilon,E)$.  By compactness of
$\Sigma_{\lambda,\alpha}$, we conclude that there exists
$\delta(\varepsilon)>0$ and $n_0(\varepsilon)>0$ such that if
$|\alpha'-\alpha|<\delta(\varepsilon)$ and $E$ is at distance at
most $\delta(\varepsilon)$ of $\Sigma_{\lambda,\alpha}$, then
$$
\ln \|A^{(\lambda,\alpha',E)}_n(\theta)\|<\varepsilon n
$$
for every $\theta$ and $n>n_0(\varepsilon)$.  If $p/q$ is
sufficiently close to $\alpha$ so that $q>n_0(\varepsilon)$ and
$\Sigma_{\lambda,p/q}$ is contained in a $\delta(\varepsilon)$
neighborhood of $\Sigma_{\lambda,\alpha}$, it then follows that
$$
\ln \|A^{(\lambda,p/q,E)}_q(\theta)\|<\varepsilon q
$$
for every $E \in \Sigma_{\lambda,p/q}$, and in particular for
every $E \in \sigma_{\lambda,p/q}$.

On the other hand, by definition of $X$, $|\tr A_q| < 2-1/5q^2$ if
$E \in X$. It now follows from \eqref {phi} that $\ln
\phi(m)=o(q)$.
\end{proof}

\begin{lemma} \label {X-Y}
We have \be |N_{\lambda,p/q} (X \setminus Y)| = o(1). \ee
\end{lemma}

\begin{proof}
Note that \be | X \setminus Y | = | X \setminus \Sigma_{\lambda, \alpha} | \leq
e^{-(c-o(1))q} \ee by (\ref {inp}).  On the other hand, \be \ln \frac {d} {dE}
N_{\lambda,p/q}(E)=o(q). \ee over $X$ by Lemma \ref {o(q)} and (\ref {dnp}). Thus,
$$
\frac {d} {dE} N_{\lambda,p/q}(E) = e^{o(q)}
$$
and hence
$$
|N_{\lambda,p/q} (X \setminus Y)| = e^{-(c-o(1))q} e^{o(q)} = e^{-(c-o(1))q},
$$
from which the result follows.
\end{proof}

\begin{lemma} \label {>}
If $E \in X \cap Y$, then \be \ln \int \phi(\tilde m(\theta)) \, d\theta> \ln \int
\phi(m(\theta)) \, d\theta-o(1). \ee
\end{lemma}

We will give the proof of this lemma in the next section.

We can now easily conclude (\ref {int_Y}) and thus Theorem \ref {real} and the Main
Theorem.  We have
\begin{align*}
\int_Y \frac {d} {dE} N_{\lambda, \alpha} (E) \, dE &\geq \frac {1} {2 \pi} \int_{X \cap
Y} \int \phi(\tilde
m (\theta)) \, d\theta \, dE\\
\nonumber &\geq (1-o(1)) \frac {1} {2 \pi} \int_{X \cap Y}
\int \phi(m (\theta)) \, d\theta \, dE\\
\nonumber
&\geq (1-o(1)) |N_{\lambda, p/q}(X \cap Y)|\\
\nonumber
&\geq 1-o(1),
\end{align*}
where the first inequality is due to Theorem \ref {Y}, the second is due to Lemma \ref
{>}, the third is due to \eqref {dnp} and absolute continuity of the IDS in the periodic
case, and the fourth is due to Lemmas \ref {NX} and \ref {X-Y}.

\section{Proof of Lemma \ref {>}}

We keep the notation from the previous section.

Since $E \in X$, $\rho \in P_q$.  Let $a$ and $b$ be as in the
definition of $P_q$.  Let us show that for every $\theta$,
\begin{equation}\label {avera}
\frac {\phi(\tilde m(\theta))+\phi(\tilde m(\theta+b q \alpha))}
{2}>(1-o(1))\phi(m(\theta)),
\end{equation}
which easily implies Lemma \ref {>}.  The estimate \eqref{avera}
is obvious when $\phi(\tilde m(\theta)) > 2 \phi(m(\theta))$, so
we will assume from now on that $\phi(\tilde m(\theta)) \leq 2
\phi(m(\theta)) \leq e^{o(q)}$.

Choose $B(\theta) \in \SL(2,\R)$ with $B(\theta) \cdot i=m(\theta)$. Then, since
$A(\theta) \cdot m(\theta) = m(\theta + p/q)$, it follows that
$$
B(\theta+p/q)^{-1} A(\theta) B(\theta) \cdot i = i
$$
and hence
\begin{equation}\label{athetarep}
A(\theta)=B(\theta+p/q)R_{\psi(\theta)} B(\theta)^{-1}.
\end{equation}
By Lemma \ref {o(q)}, $\ln \|B(\theta)\|=o(q)$ (recall that
$\phi(m(\theta)) = \frac12 \|B(\theta)\|^2_\hs$ if $B(\theta)$
takes $i$ to $m(\theta)$). We have
$$
\prod_{i=q-1}^0 R_{\psi(\theta+i p/q)}=B(\theta) A_q(\theta) B(\theta)^{-1}=
R_{\varepsilon \rho(\theta)},
$$
where $\varepsilon$ is either $1$ or $-1$.
The first identity follows from \eqref{athetarep} and the second from the definition of
$\rho(\theta)$.

\begin{lemma} \label {orbit}

We have $\|B(\theta)^{-1} \tilde A_{b q}(\theta)
B(\theta)-R\|=O(e^{-cq/4})$ where $R=R_{\varepsilon/4}$ or
$R=R_{-\varepsilon/4}$
according to whether $a=1$ or $a=3$ modulo $4$.

\end{lemma}

\begin{proof}

Write \be \tilde A_k(\theta)=\prod_{i=k-1}^0 A(\theta+i\alpha)=\prod_{i=k-1}^0
B(\theta+(i+1)p/q) Q_i B(\theta+i p/q)^{-1}. \ee That is,
\begin{align*}
Q_i & = B(\theta+(i+1)p/q)^{-1} A(\theta+i\alpha) B(\theta+i p/q) \\
& = B(\theta+(i+1)p/q)^{-1} A(\theta+ip/q) B(\theta+i p/q) + \\
& \quad + B(\theta+(i+1)p/q)^{-1} [A(\theta+i\alpha) - A(\theta+ip/q)] B(\theta+i p/q) \\
& = R_{\psi(\theta+ip/q)} + \\
& \quad + B(\theta+(i+1)p/q)^{-1} [A(\theta+i\alpha) - A(\theta+ip/q)] B(\theta+i p/q)
\end{align*}

Thus,
$$
\|Q_i-R_{\psi(\theta+ip/q)}\| = e^{o(q)} \left( |i|
e^{-(\beta-(o(1))q} \right) e^{o(q)}
$$
and hence, for $0 \leq i<b q$,
$$
\|Q_i-R_{\psi(\theta+ip/q)}\| = O(e^{o(q) + \frac{c}{2} q
-(\beta-(o(1))q  + o(q)}) = O(e^{-(3c/2-o(1))q}).
$$
Thus $\tilde A_{bq}(\theta)=B(\theta) Q B(\theta)^{-1}$ where
$Q=\prod_{i=bq-1}^0 Q_i$ satisfies $\|Q-R_{\varepsilon b
\rho(\theta)}\|=O(e^{-(c-o(1))q})$. Moreover, $|4b \rho(\theta)-a|
\leq O(e^{-cq/4})$, giving the result.
\end{proof}

We will need an estimate in hyperbolic geometry, which was already used in a
similar context in \cite {AK2}.

\begin{lemma} \label {midpoint}

Let $z_1,z_2,z_3$ lie in the same hyperbolic geodesic of $\H$, with $z_3$
the midpoint between $z_1$ and $z_2$.  Then $\phi(z_1)+\phi(z_2) \geq 2
\phi(z_3)$.

\end{lemma}

\begin{proof}

Let $2 \ln k$ be the hyperbolic distance between $z_1$ and $z_2$. We may assume that
$z_1$, $z_2$ and $z_3$ are obtained from $ki$, $i/k$ and $i$ by applying
$$
A = \left (\bm a&b\\c&d \em \right ) \in \SL(2,\R).
$$
Then,
\begin{align*}
\phi(z_1)+\phi(z_2) & = \phi \left( A \left( \begin{array}{cc} k^{1/2} & 0 \\ 0 &
k^{-1/2} \end{array} \right) \cdot i \right) + \phi \left( A \left( \begin{array}{cc}
k^{-1/2} & 0 \\ 0 & k^{1/2} \end{array} \right) \cdot i \right) \\
& = \frac12 \left\| A \left( \begin{array}{cc} k^{1/2} & 0 \\ 0 & k^{-1/2} \end{array}
\right) \right\|_\hs^2 + \frac12 \left\| A \left( \begin{array}{cc} k^{-1/2} & 0 \\ 0 &
k^{1/2} \end{array} \right) \right\|_\hs^2 \\
& = \frac12 (ka^2 + k^{-1}b^2 + kc^2 + k^{-1}d^2 ) + \frac12 (k^{-1}a^2 + kb^2 +
k^{-1}c^2 + kd^2) \\
& = \frac12 (k + k^{-1}) (a^2 + b^2 + c^2 + d^2) \\
& = \frac12 \frac{1 + k^2}{k} \|A\|_\hs^2 \\
& = \frac{1 + k^2}{k} \phi (z_3),
\end{align*}
and hence $\phi(z_1)+\phi(z_2) \geq 2 \phi(z_3)$, as desired.
\end{proof}

We can now conclude.  Since we have $\phi(m(\theta))=e^{o(q)}$, we obtain \be \label {1}
|\ln \phi(\tilde m(\theta+b q\alpha))-\ln \phi(B(\theta)R_{1/4}B(\theta)^{-1} \cdot
\tilde m(\theta))|=o(1), \ee by Lemma \ref {orbit}. More precisely, if $\tilde B(\theta)$
takes $i$ to $\tilde m(\theta)$, then
\begin{align*}
\phi(\tilde m(\theta+b q\alpha)) & = \phi(\tilde A_{bq} (\theta) \cdot \tilde m(\theta))
\\
& = \phi(\tilde A_{bq} (\theta) \tilde B(\theta) \cdot i) \\
& = \frac12 \| \tilde A_{bq} (\theta) \tilde B(\theta) \|_\hs^2
\end{align*}
and
\begin{align*}
\phi(B(\theta)R_{1/4}B(\theta)^{-1} \cdot \tilde m(\theta)) & =
\phi(B(\theta)R_{1/4}B(\theta)^{-1} \tilde B(\theta) \cdot i) \\
& = \frac12 \| B(\theta)R_{1/4}B(\theta)^{-1} \tilde B(\theta) \|_\hs^2
\end{align*}

Notice that
\begin{align*}
\| \tilde A_{bq} (\theta) \tilde B(\theta) & - B(\theta)R_{1/4}B(\theta)^{-1}
\tilde B(\theta) \|_\hs  = \\
& = \| B(\theta) [ B(\theta)^{-1} \tilde A_{bq}(\theta) B(\theta) - R_{1/4} ]
B(\theta)^{-1} \tilde B(\theta) \|_\hs \\
& \le \| B(\theta) \|_\hs \, \| B(\theta)^{-1} \tilde A_{bq}(\theta) B(\theta) - R_{1/4}
\|_\hs \, \| B(\theta)^{-1} \tilde B(\theta) \|_\hs \\
& \le e^{o(q)} e^{-cq/4} e^{o(q)} \\
& = e^{-cq/4 + o(q)}.
\end{align*}

The triangle inequality shows that
$$
\| A - B \| \le \varepsilon \Rightarrow \ln \frac{\|A\|}{\|B\|} \le \ln ( 1 +
\varepsilon)
$$
whenever $\| B \| \ge 1$.

It follows that
\begin{align*}
\ln \phi(\tilde m(\theta + & b q\alpha)) -\ln \phi(B(\theta)R_{1/4}B(\theta)^{-1} \cdot
\tilde m(\theta)) = \\
& = \ln \frac12 \| \tilde A_{bq} (\theta) \tilde B(\theta) \|_\hs^2 - \ln \frac12 \|
B(\theta)R_{1/4}B(\theta)^{-1} \tilde B(\theta) \|_\hs^2 \\
& = 2 \ln \frac{ \| \tilde A_{bq} (\theta) \tilde B(\theta) \|_\hs}{ \|
B(\theta)R_{1/4}B(\theta)^{-1} \tilde B(\theta) \|_\hs} \\
& = 2 \ln \left( 1 + e^{-cq/4 + o(q)} \right) \\
& = o(1),
\end{align*}
which is \eqref{1}.

Let us show that the points $z_1=\tilde m(\theta)$, $z_2=B(\theta) R_{1/4} B(\theta)^{-1}
\cdot \tilde m(\theta)$ and $z_3=m(\theta)$ are as in Lemma \ref {midpoint} (recall that
$m(\theta)=B(\theta) \cdot i$). Since $B(\theta)$ preserves hyperbolic distance, it is
enough to show this for the points
$$
B(\theta)^{-1} \cdot \tilde m(\theta), \; R_{1/4} B(\theta)^{-1} \cdot \tilde m(\theta),
\; i.
$$
Map $\H$ to $\D$ and observe that $i$ gets mapped to $0$ and the other two points to
diametrically opposite points. Thus, these points lie on a hyperbolic geodesic and $0$ is
the midpoint.

So by Lemma \ref {midpoint}, we have \be \label {2} \frac {\phi(\tilde
m(\theta))+\phi(B(\theta)R_{1/4}B(\theta)^{-1} \tilde m(\theta))} {2} \geq
\phi(m(\theta)). \ee Using (\ref {1}) and (\ref {2}) one gets (\ref {avera}).  This
completes the proof of Lemma \ref {>}.


\begin{thebibliography}{10}

\bibitem{AA} S.\, Aubry and G.\, Andr\'e,
Analyticity breaking and Anderson localization in incommensurate
lattices, \textit{Ann. Israel Phys. Soc.} \textbf{3} (1980),
133--164

\bibitem{AJ1} A.\ Avila and S.\ Jitomirskaya, The Ten Martini Problem, to appear in \textit
{Ann.\, of Math.}


\bibitem{AK1} A.\ Avila and R.\ Krikorian,  Reducibility or non-uniform
hyperbolicity for quasiperiodic Schr\"odinger cocycles,
\textit{Ann.\, of Math.} \textbf{164} (2006), 911--940

\bibitem{AK2} A.\ Avila and R.\ Krikorian,  Quasiperiodic $\mathrm{SL}(2,\R)$ cocycles, In preparation

\bibitem{AMS} J.\ Avron, P.\ van Mouche, and B.\ Simon, On the measure of the spectrum for the
almost Mathieu operator, \textit{Comm.\ Math.\ Phys.} \textbf{132}
(1990), 103--118

\bibitem{AS1} J.\ Avron and B.\ Simon, Singular continuous spectrum for a class of almost periodic
Jacobi matrices, \textit{Bull.\ Amer.\ Math.\ Soc.} \textbf{6} (1982), 81--85

\bibitem{AS2} J.\ Avron and B.\ Simon,
Almost periodic Schr\"odinger operators. II: The integrated density
of states, \textit {Duke Math. J.} \textbf {50} (1983), 369--391

\bibitem{BS} J.\ Bellissard and B.\ Simon, Cantor spectrum for the almost
Mathieu equation, \textit{J.\, Funct.\, Anal.} \textbf{48} (1982), 408--419

\bibitem{BJ} J.\ Bourgain and S.\ Jitomirskaya, Continuity of the Lyapunov
exponent for quasiperiodic operators with analytic potential, \textit {J.\, Statist.\,
Phys.} \textbf {108}  (2002), 1203--1218

\bibitem{CEY} M.D.\ Choi, G.A.\ Elliott, N.\ Yui, Gauss polynomials and the
rotation algebra, \textit{Invent.\, Math.} \textbf{99} (1990), 225--246

\bibitem{D} D.\, Damanik, Lyapunov exponents and spectral analysis of ergodic Schr\"odinger operators:
A survey of Kotani theory and its applications, \textit{Spectral
Theory and Mathematical Physics: a Festschrift in honor of Barry
Simon's 60th Birthday}, 539--563, Proc. Sympos. Pure Math., 76,
Part 2, Amer. Math. Soc., Providence, RI, 2007

\bibitem{DS} P.\, Deift and B.\ Simon,
Almost periodic Schr\"odinger operators. III: The absolutely continuous
spectrum in one dimension,
\textit {Comm.\, Math.\, Phys.} \textbf {90} (1983), 389--411

\bibitem{GS1} M.\, Goldstein and W.\, Schlag,
H\"older continuity of the integrated density of states for
quasi-periodic Schr\"odinger equations and averages of shifts of
subharmonic functions, \textit{Ann. of Math.} \textbf{154} (2001),
155--203

\bibitem{GS2} M.\, Goldstein and W.\, Schlag,
Fine properties of the integrated density of states and a
quantitative separation property of the Dirichlet eigenvalues, to
appear in \textit{Geom. Funct. Anal.}

\bibitem{G} A.\ Gordon, On the point spectrum of the one-dimensional Schr\"odinger operator,
\textit{Usp.\ Math.\ Nauk.} {\bf 31} (1976), 257--258

\bibitem{GJLS} A.\ Gordon, S.\, Jitomirskaya, Y.\, Last, and B. Simon,
Duality and singular continuous spectrum in the almost Mathieu equation,
\textit {Acta Math.} \textbf {178} (1997), 169--183

\bibitem{J} S.\ Jitomirskaya, Metal-insulator transition for the almost Mathieu operator,
\textit{Ann.\ of Math.} \textbf{150} (1999), 1159--1175

\bibitem{JM} R.\, Johnson and J.\, Moser, The rotation number for almost
periodic potentials, \textit {Comm.\ Math.\ Phys.} \textbf {84} (1982), 403--438

\bibitem{k2} S.\ Kotani, Generalized Floquet theory for stationary Schr\"odinger operators in one
dimension, \textit{Chaos Solitons Fractals} \textbf{8} (1997),
1817--1854

\bibitem{L} Y.\ Last, A relation between a.c. spectrum of ergodic Jacobi matrices and the spectra
of periodic approximants, \textit{Comm.\ Math.\ Phys.} \textbf{151}
(1993), 183--192

\bibitem{L1} Y.\ Last, Zero measure spectrum for the almost Mathieu operator, \textit{Commun.\
Math.\ Phys.} \textbf{164} (1994), 421--432

\bibitem{L2} Y.\ Last, Spectral theory of Sturm-Liouville operators on infinite intervals: a review
of recent developments, \textit {Sturm-Liouville theory}, 99--120, Birkh\"auser, Basel,
2005

\bibitem{P} J.\ Puig, Cantor spectrum for the almost Mathieu operator,
\textit{Comm.\, Math.\, Phys.} \textbf{244}  (2004), 297--309

\bibitem{s3} B.\, Simon, Almost periodic Schr\"odinger operators: a review,
\textit{Adv.\, Appl.\, Math.} \textbf{3} (1982), 463--490

\bibitem{s1} B.\ Simon, Kotani theory for one dimensional stochastic Jacobi matrices, \textit{Comm.\
Math.\ Phys.} \textbf{89} (1983), 227--234

\bibitem{s2} B.\ Simon, Schr\"odinger operators in the twenty-first century, \textit{Mathematical physics
2000}, 283--288, Imp. Coll. Press, London, 2000

\bibitem{T} D.\, Thouless, Bandwidths for a quasiperiodic tight binding
model, \textit {Phys.\, Rev.\, B} \textbf{28} (1983), 4272--4776

\end{thebibliography}
\end{document}